\documentclass[11pt,reqno]{amsproc}
\usepackage[margin=1in]{geometry}
\usepackage{amsmath, amsthm, amssymb}
\usepackage{times, esint,stackrel}
\usepackage{enumitem}
\usepackage{color}  
\usepackage[colorlinks=true]{hyperref}
\usepackage[color=yellow]{todonotes}
\usepackage{mathtools}
\usepackage{etoolbox}
\usepackage{mathrsfs}
\usepackage{framed}

\makeatletter
\patchcmd\@thm
  {\let\thm@indent\indent}{\let\thm@indent\noindent}%
  {}{}
\makeatother

\usepackage{etoolbox}
\expandafter\patchcmd\csname\string\proof\endcsname
  {\normalparindent}{0pt }{}{}

\newtheorem*{thm*}{Theorem}
\newcommand{\pa}{\partial}
\newcommand{\la}{\label}
\newcommand{\fr}{\frac}
\newcommand{\na}{\nabla}
\newcommand{\be}{\begin{equation}}
\newcommand{\ee}{\end{equation}}
\newcommand{\bea}{\begin{eqnarray}}
\newcommand{\eea}{\end{eqnarray}}

\newtheorem{thm}{Theorem}
\newtheorem{prop}{Proposition}
\newtheorem{lemma}{Lemma}
\newtheorem{cor}{Corollary}

\theoremstyle{definition}
\newtheorem{rem}{Remark}

\newcommand{\ve}{{\varepsilon}}

\newcommand{\rmd}{{\rm d}}

\def\wc{\rightharpoonup}

\newcommand{\bq}{\begin{equation}}
\newcommand{\eq}{\end{equation}}
\newcommand{\bqa}{\begin{eqnarray*}}
\newcommand{\eqa}{\end{eqnarray*}}

\def\XXint#1#2#3{{\setbox0=\hbox{$#1{#2#3}{\int}$ }
\vcenter{\hbox{$#2#3$ }}\kern-.6\wd0}}

\begin{document}

\title{Inviscid limit of vorticity distributions in Yudovich class}

\author{Peter Constantin}
\address{Department of Mathematics, Princeton University, Princeton, NJ 08544}
\email{const@math.princeton.edu}

\author{Theodore D. Drivas}
\address{Department of Mathematics, Princeton University, Princeton, NJ 08544}
\email{tdrivas@math.princeton.edu}

\author{Tarek M. Elgindi}
\address{Department of Mathematics, UC San Diego}
\email{telgindi@ucsd.edu}

 \maketitle
 
\date{today}


\begin{abstract}
We prove that given initial data $\omega_0\in L^\infty(\mathbb{T}^2)$, forcing $g\in L^\infty(0,T; L^\infty(\mathbb{T}^2))$, and any $T>0$, the solutions $u^\nu$ of Navier-Stokes converge strongly in $L^\infty(0,T;W^{1,p}(\mathbb{T}^2))$ for any $p\in [1,\infty)$ to the unique Yudovich weak solution  $u$ of the Euler equations.
A consequence is that vorticity distribution functions converge to their inviscid counterparts. As a byproduct of the proof, we establish continuity of the Euler solution map for Yudovich solutions in the $L^p$ vorticity topology.  The main tool in these proofs is a uniformly controlled loss of regularity property of the linear transport by Yudovich solutions. Our results provide a partial foundation for the Miller--Robert statistical equilibrium theory of vortices as it applies to slightly viscous fluids.
\end{abstract}

\section{Introduction }

In this paper we discuss the connection between Yudovich solutions of the Euler equations
\be\label{eu}
\partial_t \omega+ u\cdot \nabla \omega= g, 
\ee
with bounded forcing $g\in L^\infty(0,T;L^\infty(\mathbb{T}^2))$,
and initial data
\be
\omega(0)=\omega_0\in L^\infty(\mathbb{T}^2),
\la{ideu}
\ee
and the vanishing viscosity limit ($\lim_{\nu\to 0}$) of solutions of the Navier-Stokes equations,
\be\la{nse}
\partial_t \omega^\nu+ u^\nu\cdot \nabla \omega^\nu= \nu \Delta \omega^\nu+ g, 
\ee
with initial data
\be
\omega^{\nu}(0) = \omega^\nu_0\in L^{\infty}(\mathbb{T}^2),
\la{idnse}
\ee
and the same forcing $g$.  We consider uniformly bounded initial data
\be
\sup_{\nu>0}\|\omega^\nu_0\|_{L^{\infty}(\mathbb T^2)}\le \Omega_{0,\infty}<\infty.
\la{supid}
\ee
The solutions of (\ref{eu}), (\ref{ideu}), (\ref{nse}), (\ref{idnse}) are uniformly bounded in $L^{\infty}(\mathbb T^2)$:
\be
\sup_{\nu\ge  0}\sup_{0\le t\le T}\|\omega^{\nu}(t)\|_{L^{\infty}(\mathbb T^2)}
\le \Omega_{\infty} = \Omega_{0,\infty} + \int_0^T\|g(t)\|_{L^{\infty}(\mathbb T^2)}dt.
\la{omegafty}
\ee
This bound is valid in $\mathbb T^2$ or $\mathbb R^2$ but is not available if boundaries are present or in 3D. The bound will be used repeatedly below.

We are interested in the small viscosity behavior of vorticity distribution function $\pi_{\omega^{\nu}(t)}(\rmd y)$ defined by
\be\label{vordens}
\int  f(y) \pi_{\omega^\nu(t)}(\rmd y) = \int f(\omega^\nu(t,x)) \rmd x,
\ee
for all continuous functions (observables) $f$.  If $\omega_0^\nu \to \omega_0$  we  prove that the  distributions convergence 
\be \label{distconv}
\pi_{\omega^\nu(t)}(\rmd y)\ \ \xrightarrow[]{\nu \to 0} \ \  \pi_{\omega(t)}(\rmd y) = \pi_{\omega_0}(\rmd y),
\ee
where the time invariance of the vorticity distribution function for the Euler equations follows from Lagrangian transport $\omega(t)= \omega_0\circ X_t^{-1}$ and volume preservation of the homeomorphism $A_t=X_t^{-1}$.

The statement \eqref{distconv} is a consequence of the strong convergence of the vorticity in $L^\infty(0,T;L^p(\mathbb{T}^2))$ for all $p\in [1,\infty)$ and for any $T>0$.  We prove this fact here, extending previous work for vortex patch solutions with smooth boundary \cite{CW1}, and removing  additional assumptions on the Euler path \cite{CW2}.  Implications of our result for equilibrium theories of decaying two dimensional turbulence \cite{Miller,Robert} are briefly discussed at the end of this paper. Our main result is the following.
  
\begin{thm}\label{Thm}
Let $\omega$ be the unique Yudovich weak solution of the Euler equations with initial data $\omega_0\in L^\infty(\mathbb{T}^2)$ and forcing $g\in L^\infty(0,T;L^\infty(\mathbb{T}^2))$.  Let $\omega^\nu$ be the solution of the Navier-Stokes equation with the same forcing and initial data $\omega^\nu_0 \to \omega_0$ strongly in $L^2(\mathbb{T}^2)$. 
Then, for any $T>0$ and $p\in [1,\infty)$, the inviscid limit  $\omega^\nu\to \omega$  holds strongly in $L^\infty(0,T; L^p(\mathbb{T}^2))$:
\be
\lim_{\nu\to 0}\sup_{0\le t\le T}\|\omega^\nu(t)-\omega (t)\|_{L^p(\mathbb T^2)} = 0.
\la{invp}
\ee
Consequently, the distributions converge,
\be
\lim_{\nu\to 0} \pi_{\omega^\nu(t)}(\rmd y) = \pi_{\omega_0}(\rmd y),
\ee
for all $ t\in [0,T]$.
\end{thm}

\begin{rem}
There are several senses in which this theorem is sharp. First, there can be no infinite time result as the Euler solution is conservative and the Navier-Stokes solution is dissipative.  This is obvious if we consider the stationary solutions $\omega_0(x)=\sin(Nx)$ and $g=0$. Secondly, there can be no rate without additional regularity assumptions on $\omega_0$, as is the case for the heat equation.    Thirdly, there can be no strong convergence in $L^\infty$ because $\omega_0$ may not be continuous while $\omega^\nu$ is smooth for any $t>0$. And, finally there can be no strong convergence for $p>1$ in domains with boundaries, if the boundary condition of the Navier-Stokes solutions is no slip, and the Euler solution has non-vanishing tangential velocity at the boundary, in other words, if there are boundary layers \cite{Kelliher}.
\end{rem}

\begin{rem}
One implication of theorem \ref{Thm} is that the dissipation of convex functions of vorticity must vanish,
\be\label{enstdiss}
\lim_{\nu \to 0} \nu \int_0^T \int_{\mathbb{T}^2} f''(\omega^\nu) |\nabla \omega^\nu|^2 \rmd x \rmd t = 0.
\ee
In the special case when $f(x)= |x|^2/2$, the above is the enstrophy dissipation (palenstrophy).  In fact, it was proved by Eyink that anomalous enstrophy dissipation requires that $\omega_0\notin L^2(\mathbb{T}^2)$ \cite{Eyink01,HML06}.  The idea is that, if  $\omega_0 \in L^2(\mathbb{T}^2)$, the enstrophy remains uniformly-in-$\nu$ bounded since it is non-increasing under the Navier-Stokes evolution. Applying the Aubin-Lions lemma yields weak convergence on subsequences to $\omega$, a weak solution of the Euler equations  (possibly non-unique).  Thus $\omega^\nu\to \omega$ in $C(0,T; w-L^2(\mathbb{T}^2))$. Moreover, for such initial data, all weak Euler solutions can be shown to be renormalized in the sense of DiPerna-Lions and hence conservative \cite{CS15}.  Thus, by weak lower semi-continuity of the $L^2$ norm, the Navier-Stokes enstrophy balance implies also that norms converge and hence the convergence is strong in $L^2$, pointwise in time, i.e.  
$\omega^\nu(t)\to \omega$ in $L^2(\mathbb{T}^2)$ for each $t\in [0,T]$.  In fact, whenever the vorticity converges weakly to a conservative weak Euler solution, one has strong convergence and there can be no anomaly.  The convergence can be made uniform in time. This proof using compactness, however, inherently gives a qualitative statement and one cannot extract information about rates of convergence.  On the other hand, our proof is quantitative. Specifically, given information on, say, the spectrum of the initial vorticity at high wavenumber, one can obtain a rate of convergence. One class of examples which we discuss in corollary \ref{cor} concerns vorticity in the space $\omega_0\in L^\infty \cap B^s_{p,\infty}$ for $s>0$.  However, more generally, for any $\omega_0\in L^\infty$ our proof provides a computable rate of convergence depending only on $\omega_0$, independent of the particular subsequence $\{\nu_n\}_{n\geq 0}$. 
\end{rem}

A corollary of the proof of theorem \ref{Thm} and lemma \ref{chem} is the continuity of the Yudovich solution map $\omega(t)= S_{t}(\omega_0)$ in the $L^p$ topology when restricted to fixed balls in $L^\infty$. 
\begin{cor}\label{contcor}
Fix $T>0$, $\omega_0 \in L^\infty(\mathbb{T}^2)$ and $\ve>0$.  There exists $\delta:= \delta(\omega_0,T, \ve)$ so that for all $\omega \in L^\infty(\mathbb{T}^2)$, 
\be
\|\omega- \omega_0\|_{L^p(\mathbb T^2)}<\delta, \qquad \text{implies} \qquad  \sup_{t\in[0,T]} \|S_t(\omega)  - S_t(\omega_0)\|_{L^p(\mathbb T^2)} <\ve.
\ee
\end{cor}

The proof of Theorem \ref{Thm} is based on the fact that linear transport by 
Yudovich solutions has a short time uniformly controlled loss of regularity: it maps bounded sets in $W^{1,p}, \; p>2$ to bounded sets in $H^1$, uniformly in viscosity. 
More precisely, we consider the Yudovich solutions $\omega (t)$ and $\omega^\nu(t)$ of the Euler and Navier-Stokes equations with initial data $\omega_0\in L^{\infty}$ and denote their corresponding velocities by  $u(t)$ and, respectively,  $u^\nu(t)$. We take a sequence of regularizations $\omega_{0,n}\in W^{1,\infty}$ of $\omega_0$ which is uniformly bounded in $W^{1,p}, \; p>2$ and is such that $\omega_{0,n}\rightarrow \omega_0$ strongly in  $L^2$. We let $\omega_n(t)$ be the unique solutions of the linear transport problems \[\partial_t \omega_n+u\cdot\nabla\omega_n=0\] and respectively  $\omega_n^\nu(t)$ of  \[\partial_t \omega_n^\nu+u^\nu\cdot\nabla\omega_n^\nu=\nu\Delta\omega_n^\nu.\] On one hand, $\omega_n(t)$ remains close to $\omega(t)$ and $\omega_n^\nu(t)$ remains close to $\omega^\nu(t)$ in $L^p$ spaces because linear transport bu Yudovich velocities is clearly bounded in $L^p$. The essential additional ingredient we show is a controlled loss of regularity: $\omega_n(t)$ and $\omega_n^\nu(t)$ are bounded in $H^1$ on a short time interval by their initial norms in $W^{1,p}$, $p>2$. This uses the fact that $\nabla u$ and $\nabla u^\nu$ are exponentially integrable. The rest of the proof rests on these observations as well as energy estimates and a time splitting. 

In the direction of propagating regularity, we also prove the fact that if additional smoothness is assumed on the data then some degree of fractional smoothness in $L^p$ can be propagated uniformly in viscosity.  We consider the unforced case $g=0$ and we fix initial data $\omega_0^\nu = \omega_0$ for simplicity, the natural extension being straightforward.
\begin{prop}\label{propReg}
Suppose $\omega_0\in (L^\infty\cap B^{s}_{p,\infty})(\mathbb{T}^2)$ for some  $s>0$ and some $p\geq 1$. Then the solutions of the Navier-Stokes equations satisfy $\omega^\nu(t)\in(L^\infty\cap B^{s(t)}_{p,\infty})(\mathbb{T}^2)$ uniformly in $\nu$, where 
\[
s(t) = {s}\exp(-Ct\|\omega_0\|_{L^{\infty}(\mathbb T^2)})
\]
for some universal constant $C>0$.
\end{prop}
The proof of Proposition \ref{propReg} relies on the fact that the velocity is log-Lipschitz uniformly in $\nu$ and shows that the exponential estimate with loss of \cite{BC94} holds uniformly in viscosity.  Our proof  uses the stochastic Lagrangian representation formula of \cite{CI08}:
\be\label{traj}
\rmd X_t(x) = u^\nu(X_t(x),t)\rmd t+ \sqrt{2\nu} \ \rmd W_t, \qquad X_0(x)=x,
\ee
 yielding the representation formula
\be \label{stochRep}
\omega^\nu(t) = \mathbb{E}\left[\omega_0\circ A_t \right]
\ee
where back-to-labels map is defined as $A_t= X_t^{-1}$.
 The noisy Lagrangian picture allows for a nearly direct application of the Theorems and proofs of \cite{BC94,BCD11} to the viscous case.  We remark that the uniform Sobolev regularity can be established by similar arguments; if $\omega_0\in (L^\infty\cap W^{s,p})(\mathbb{T}^2)$ then  $\omega^\nu(t)\in(L^\infty\cap W^{s(t),p})(\mathbb{T}^2)$ with uniformly bounded norms.  
 
The uniform regularity of Proposition \ref{propReg}  is used  to deduce
\begin{cor}\label{cor}
Let $\omega_0\in (L^\infty\cap B^{s}_{2,\infty})(\mathbb{T}^2)$ with $s>0$ and let  $\omega$ and $\omega^\nu$ solve respectively (\ref{eu})
and (\ref{nse}), with the same initial data $\omega^\nu_0=\omega_0$.  Then the $L^p$ convergence of vorticity, for any $p\in[1,\infty)$ and any finite time $T>0$, occurs at the rate
\be
\sup_{t\in [0,T]}\|\omega^\nu(t)- \omega(t)\|_{L^p(\mathbb{T}^2)} \lesssim    (\nu T)^{\frac{ s\exp(-2CT\|\omega_0\|_{\infty})}{p(1+ s \exp(-CT\|\omega_0\|_{\infty}))}- },
\ee
with the universal constant $C>0$ in Proposition \ref{propReg}.
\end{cor}

\begin{rem}
Recently, the estimate with loss of \cite{BC94} was sharpened for fixed $p\in(1,\infty)$ in \cite{BN19} where it is shown that the propagated regularity decays inversely with time rather than exponentially, i.e. $ \tilde{s}(t) = {s}/{(1+Ct p s  )}$  for some universal constant $C>0$. See Corollary 1.4 of  \cite{BN19}. This improvement is accomplished by taking greater advantage of the uniform exponential integrability of the velocity gradient stated in Lemma \ref{explemm} below.
The stochastic representation can also be used to show uniform boundedness of the vorticity in $\omega^\nu(t)\in(L^\infty\cap B^{\tilde{s}(t)}_{p,\infty})$ as was done in Proposition  \ref{propReg}.  We omit details here, which are straightforward extensions of the proofs of \cite{BN19}.  This extension can lead to an improved rate in Corollary \ref{cor}.
\end{rem}

Corollary \ref{cor} applies in particular to the to inviscid limits of vortex patches with non-smooth boundary. Indeed, lemma 3.2 of \cite{CW2} shows that  if $\omega_0= \chi_\Omega$ is the characteristic function of a bounded domain whose boundary has  box-counting (fractal) dimension $D$ not larger than the dimension of space $d=2$, i.e. $d_F(\partial \Omega ):= D<2$, then $\omega_0\in B_{p,\infty}^{(2-D)/p}(\mathbb{T}^2)$.  Proposition \ref{propReg} then shows that some degree of fractional Besov regularity of the solution $\omega^\nu(t)$ is retained uniformly in viscosity for any finite time $T<\infty$ and corollary \ref{cor} provides a rate depending only $D, T$ and $p$ at which the vanishing viscosity limit holds, removing therefore the need for the additional assumptions on the solution imposed in \cite{CW2}.

\section{Proof }

\begin{proof}[Proof of Theorem \ref{Thm}]
It suffices to prove that 
\be
\lim_{\nu\to 0} \sup_{t\in [0,T]} \|\omega^\nu(t)-\omega(t)\|_{L^2(\mathbb{T}^2)}=0.
\ee 
Indeed, convergence in $L^p$ for any  $p\in [2,\infty)$ then follows from interpolation and boundedness in $L^\infty$: 
\begin{align}
\|\omega^\nu(t)-\omega(t)\|_{L^p(\mathbb{T}^2)} &\leq 2\Omega_{\infty}^{\frac{p-2}{p}}\|\omega^\nu(t)-\omega(t)\|_{L^2(\mathbb{T}^2)}^{\frac{2}{p}}.
\end{align}
In order to  establish strong $L^\infty_tL^2_x$ convergence for arbitrary finite times $T$, it is enough to the convergence for a short time which depends only on a uniform $L^{\infty}$ bound on the initial vorticity:
\begin{prop}\label{prop}
Let $\omega$ and $\omega^\nu$ solve (\ref{eu}) and (\ref{nse}) respectively, with initial data (\ref{ideu}) and (\ref{idnse}). Assume that the Navier-Stokes initial data converge uniformly in $L^{2}(\mathbb T^2)$
\be
\lim_{\nu\to 0}\|\omega^\nu_0 - \omega_0\|_{L^2(\mathbb T^2)} = 0.
\la{limidl2}
\ee 
Assume also that there exists a contant $\Omega_{\infty}$ such that the initial data are uniformly bounded in $L^{\infty}(\mathbb T^2)$:
\be
\sup_{\nu>0}\|\omega_0^{\nu}\|_{L^{\infty}(\mathbb T^2)}\le \Omega_{\infty}.
\la{unifid}
\ee
Then there exists a constant $C_*$ such that the vanishing viscosity limit 
holds
\be
\lim_{\nu\to 0} \sup_{t\in [0,T_*]} \|\omega^\nu(t)-\omega(t)\|_{L^2(\mathbb{T}^2)}=0
\ee
on the time interval $[0, T_*]$ where
\be
T_*= (C_*  \Omega_{\infty})^{-1}.
\la{tstar}
\ee
\end{prop}

Once this proposition is established, the proof of theorem \ref{Thm} follows by dividing the time interval $[0,T]$ in  subintervals
$$
[0,T] = [0,T_*] \cup [T_*, 2T_*] \cup\cdots
$$
where $T_*$ is determined from the uniform bound (\ref{omegafty}), 
and applying proposition \ref{prop} to each interval, with initial data $\omega(nT_*)$, and respectively $ \omega^\nu( nT_*)$.  As there is no required rate of convergence for the initial data in proposition \ref{prop}, theorem \ref{Thm} follows.

\begin{proof}[Proof of Proposition \ref{prop}]
We introduce functions $\omega_\ell$ and $\omega_\ell^\nu$ which
are the unique solutions of the following \emph{linear} problems. We fix $\ell>0$ and let
\begin{align}\label{omell}
\partial_t \omega_\ell + u\cdot \nabla \omega_\ell &=\varphi_\ell *g , \qquad\qquad\quad \ \ \ \omega_\ell(0)=\varphi_\ell *{\omega}_0,\\
\partial_t \omega_\ell^\nu+ u^\nu\cdot \nabla \omega_\ell^\nu &=\nu \Delta \omega_\ell^\nu+\varphi_\ell *g , \qquad \omega_\ell^\nu(0)=\varphi_\ell *{\omega}_0^\nu, \label{omellnu}
\end{align}
where $\varphi_\ell $ is a standard mollifier at scale $\ell$ and where $u$ and $u^\nu$ are respectively the unique solutions of Euler and Navier-Stokes equations.  Note that the solutions to the linear problems \eqref{omell} and \eqref{omellnu} exist globally and are unique because the Yudovich velocity field $u$ is log-Lipshitz.  We observe that we have
\begin{align*}
\|\omega^\nu(t)-\omega(t)\|_{L^2(\mathbb{T}^2)} &\leq  \|\omega(t)-\omega_\ell(t)\|_{L^2(\mathbb{T}^2)} + \|\omega^\nu(t)-\omega^\nu_\ell(t)\|_{L^2(\mathbb{T}^2)} \\
&\qquad + \|\omega^\nu_\ell(t)-\omega_\ell(t)\|_{L^2(\mathbb{T}^2)}.
\end{align*}
Because the equations for $\omega_\ell,\omega_\ell^\nu$ and, respectively $\omega,\omega^\nu$ share the same incompressible velocities, we find
\begin{align}
\|\omega(t)-\omega_\ell(t)\|_{L^2(\mathbb{T}^2)}&\leq  \|{\omega}_0-\varphi_\ell *{\omega}_0\|_{L^2(\mathbb{T}^2)} +\int_0^t  \| g(s)- \varphi_\ell * g(s)\|_{L^2(\mathbb{T}^2)}\rmd s,\\
 \|\omega^\nu(t)-\omega^\nu_\ell(t)\|_{L^2(\mathbb{T}^2)}&\leq  \|{\omega}_0^\nu-\varphi_\ell *{\omega}_0^\nu\|_{L^2(\mathbb{T}^2)} +\int_0^t \| g(s)- \varphi_\ell * g(s)\|_{L^2(\mathbb{T}^2)}\rmd s.
\end{align}
As mollification can be removed strongly in $L^p$, the two terms in the right hand sides converge to zero as $\ell,\nu\to 0$, in any order.  
It remains to show that 
\begin{align} \label{convFixell}
\lim_{\nu\to 0}\sup_{t\in [0,T_*]} \|\omega^\nu_\ell(t)-\omega_\ell(t)\|_{L^2(\mathbb{T}^2)}\to 0
\end{align}
for fixed $\ell$. In order to establish this, we use two auxilliary results. The first one is a general statement about the Biot-Savart law in dimension two.

\begin{lemma}\label{explemm}
Let $\omega\in L^\infty(\mathbb{T}^2)$ and let $u$ be obtained from $\omega$ by the Biot-Savart law 
\[
u= K[\omega] = \nabla^\perp (-\Delta)^{-1} \omega.
\]
There exist constants $\gamma>0$ (nondimensional  and $C_K$ (with units of area) such that 
\be\label{expinte}
 \int_{\mathbb{T}^2} \exp\left\{ \beta |\nabla u(x)|\right\} \rmd x \leq  C_K
\ee
holds for any $\beta>0$ such that
\be
\beta\|\omega\|_{L^{\infty}(\mathbb T^2)} \le \gamma.
\la{betagamma}
\ee
\end{lemma}
\begin{proof}[Proof of Lemma \ref{explemm}]
The bound \eqref{expinte} holds due to the fact that Calderon-Zygmund operators map $L^\infty$ to BMO  \cite{Stein}, $\omega\in L^\infty \mapsto \nabla u = \nabla K[u] \in BMO$, and from  the John-Nirenberg inequality \cite{Nirenberg} for BMO functions.  We provide below a direct and elementary argument (modulo a fact about norms of singular intergal operators), for the sake of completeness.

We recall that there exists a constant $C_*$ so that for all $p\geq 2$,
\be\label{BSbnd}
\|\nabla K[v]\|_{L^p(\mathbb{T}^2)}=\|\nabla \otimes \nabla (-\Delta)^{-1} v\|_{L^p(\mathbb{T}^2)}\leq C_* p \|v\|_{L^p(\mathbb{T}^2)}.
\ee
 (See \cite{Stein}).  The dependence of \eqref{BSbnd} on $p$ is the important point. Thus,  
\begin{align} \nonumber
   \int_{\mathbb{T}^2} e^{  \beta |\nabla u|} \rmd x&= \sum_{p=0 }^\infty  \beta^p \frac{\|\nabla u\|^p_{L^p(\mathbb{T}^2)} }{p!} \leq   \sum_{p=0 }^\infty   \frac{  \left(   C_* \beta  \|\omega\|_{L^p(\mathbb{T}^2)}\right)^pp^p}{p!}\\
   &\leq  |\mathbb{T}^2|  \sum_{p=0 }^\infty   \frac{  \left(   C_* \beta \|\omega\|_{L^\infty(\mathbb{T}^2)}\right)^pp^p}{p!}.
\end{align}
This is a convergent series provided $ C_* \beta \|\omega\|_{L^\infty(\mathbb{T}^2)}<1/e$.  Indeed, this can be seen using Stirling's bound $n!\geq \sqrt{2\pi} n^{n+1/2} e^{-n} $ which yields
\be
\sum_{p=0}^\infty \frac{c^pp^p}{p!} \leq 1+\sum_{p=1}^\infty \frac{p^{-1/2}}{\sqrt{2\pi}   }(ce)^{p} \leq \frac{1}{1-ce}, \quad \text{provided}\quad  c \in[0, 1/e)
\ee
where $c:=C_* \beta \|\omega\|_{L^\infty(\mathbb{T}^2)}$. We may take thus
\be
\gamma = (2C_*e)^{-1}, \quad  C_K = 2\left |\mathbb T^2\right|.
\la{gammack}
\ee
The constant $\gamma$ depends on the Biot-Savart kernel and is nondimensional, the constant $C_K$ then is proportional to the area of the domain. 

  \end{proof}

The second auxilliary result concerns scalars transported and amplified by a velocity with bounded curl in two dimensions.
  \begin{lemma}\label{scalalemm}
Let $u:=u(x,t)$ be divergence free and $\omega:=\nabla^\perp \cdot u  \in L^\infty(0,T;L^\infty(\mathbb{T}^2))$ with
\be
\sup_{0\le t\le T}\|\omega(t)\|_{L^{\infty}(\mathbb T^2)}\le \Omega_{\infty}.
\la{omegfty}
\ee
Consider a nonnegative scalar field $\theta:=\theta(x,t)$ satisfying the differential inequality
\be\label{scalareqn}
\partial_t \theta + u \cdot \nabla \theta - \nu \Delta \theta  \leq |\nabla u |\theta + f, 
\ee
with initial data  $\theta|_{t=0}=\theta_0\in L^\infty(\mathbb{T}^2)$, and forcing $f\in L^\infty(0,T;L^\infty(\mathbb{T}^2))$. 
Let $\gamma>0$ be the constant from Lemma \ref{explemm}.  Then, for any $p>1$ and the time $T(p)= \frac{\gamma(p-1)}{2p\Omega_{\infty}}$ it holds that 
\be
\sup_{t\in [0,T(p)]} \|\theta(t)\|_{L^2(\mathbb{T}^2)} \leq C_1 \|\theta_0\|_{L^{2p}(\mathbb{T}^2)}^{p} + C_2
\ee
for some constants $C_1, C_2$ depending only on $p$, $\Omega_{\infty}$ and $\|f\|_{L^\infty(0,T;L^\infty(\mathbb{T}^2))}$.
\end{lemma}
\begin{proof}[Proof of Lemma \ref{scalalemm}] Let $p:=p(t)$ with $p(0)=p_0$ and time dependence of $p(t)$ to be specified below. Consider
\begin{align}  \nonumber
\frac{1}{2}\frac{\rmd}{\rmd t} \int_{\mathbb{T}^2} |\theta|^{2p(t)} \rmd x &= p'(t) \int_{\mathbb{T}^2} \ln|\theta| |\theta|^{2p(t)} \rmd x+ p(t)\int_{\mathbb{T}^2} |\theta|^{2p(t)-2} \theta \partial_t \theta  \rmd x \\ \nonumber
&\leq p'(t) \int_{\mathbb{T}^2} \ln|\theta| |\theta|^{2p(t)} \rmd x- p(t)\int_{\mathbb{T}^2} |\theta|^{2p(t)-2} \theta u \cdot \nabla \theta \rmd x  \\ \nonumber
&\quad + \nu p(t)\int_{\mathbb{T}^2} |\theta|^{2p(t)-2} \theta \Delta \theta \rmd x  +  p(t)\int_{\mathbb{T}^2} |\theta|^{2p(t)-2} |\nabla u| \theta^2 \rmd x  \\
&\quad + p(t)\int_{\mathbb{T}^2} |\theta|^{2p(t)-2} \theta f  \rmd x.
\end{align}
We now use the following facts
\begin{align}
\int_{\mathbb{T}^2} |\theta|^{2p-2} \theta f  \rmd x  &\leq    C\|f\|_{L^\infty(0,T;L^\infty(\mathbb{T}^2))} \|\theta\|_{2p}^{2p-1},\\
p\int_{\mathbb{T}^2} |\theta|^{2p-2} \theta u \cdot \nabla \theta \rmd x&=  \frac{1}{2} \int_{\mathbb{T}^2} u\cdot \nabla( |\theta|^{2p})  \rmd x =0,\\ 
 \nu \int_{\mathbb{T}^2} |\theta|^{2p-2} \theta \Delta \theta \rmd x &= -\nu (2p-1) \int_{\mathbb{T}^2} |\theta|^{2p-2} |\nabla  \theta|^2\rmd x \leq 0.
\end{align}
In the second equality we used the fact that the velocity is divergence free.
Altogether we find thus
\begin{align}  \nonumber
\frac{1}{2}\frac{\rmd}{\rmd t} \|\theta(t)\|_{2p(t)}^{2p(t)} \rmd x &\leq p'(t) \int_{\mathbb{T}^2} \ln|\theta| |\theta|^{2p(t)} \rmd x \\
&\qquad +  p(t)\int_{\mathbb{T}^2} |\theta|^{2p(t)}   |\nabla u|  \rmd x + p(t)  \|f\|_{L^\infty} \|\theta\|_{2p}^{2p-1}.
\end{align}
We now use the following elementary inequality: for $a\in \mathbb{R}$ and  $b>0$,
\be\label{elemIq}
ab \leq e^a+ b \ln b-b.
\ee
In fact, we use only that $ab \leq e^a+ b \ln b$. 
The inequality \eqref{elemIq} is proved via calculus and follows because the Legendre transform of the convex function $b\ln b-b+1$ is $e^a-1$. Setting $a= \beta |\nabla u|$ and $b=   \frac{1}{\beta} |\theta|^{2p}$, applying \eqref{elemIq} and Lemma \ref{explemm} we obtain
\begin{align}  \nonumber
\frac{1}{2}\frac{\rmd}{\rmd t}\|\theta(t)\|_{2p(t)}^{2p(t)} &\leq p'(t) \int_{\mathbb{T}^2} \ln|\theta| |\theta|^{2p} \rmd x+  \frac{p(t)}{\beta} \int_{\mathbb{T}^2} \ln(\beta^{-1}|\theta|^{2p}) |\theta|^{2p} \rmd x\\ \nonumber
&\qquad  +p(t) \int_{\mathbb{T}^2} e^{ \beta |\nabla u|} \rmd x+   C p(t) \|f\|_{L^\infty} \|\theta\|_{2p}^{2p-1}\\ \nonumber
&\leq\left( p'(t)+ \frac{2p(t)^2}{\beta}\right) \int_{\mathbb{T}^2} \ln|\theta| |\theta|^{2p} \rmd x + \frac{p(t)}{\beta}\ln (\beta^{-1})\|\theta(t)\|_{2p}^{2p}\\ 
&\qquad +p(t)C_K+   Cp(t) \|f\|_{L^\infty} \|\theta\|_{2p}^{2p-1},
\end{align}
 where $C_K$ is the constant from Lemma \ref{explemm} and $\beta = \fr{\gamma}{\Omega_{\infty}}$ depends on the bound for $\|\omega(t)\|_{L^{\infty}}$.  We now choose $p$ to evolve according to
 \be\label{peq}
  p'(t)=- 2\beta^{-1} p(t)^2, \ \ p(0)=p_0\quad  \implies \quad p(t)=  \frac{\beta p_0}{\beta + 2p_0 t}.
 \ee
 Note that $p(t)$ is a positive monotonically decreasing function of $t$.  Let the time $t_*$ defined by $t_*=T(p_0):=\beta(p_0-1)/2p_0$ be such that  $p(t_*)=1$.  Then $p(t)\in [1,p_0]$ for all $t\in [0,t_*]$. Note also from (\ref{peq}) that
\[
\int_0^t p(s)ds = \log\left(\fr{p_0}{p(t)}\right)^{2\beta} = \log\left(1+ \fr{2p_0 t}{\beta}\right)^{\fr2\beta}.
\]
Defining $m(t)=\frac{1}{2} \|\theta(t)\|_{2p(t)}^{2p(t)}$ and using \eqref{peq} we have the differential inequality
\be
m'(t) \leq  p(t)(C_1 m(t) + C_2) \implies C_1 m(t)+ C_2 \leq (C_1 m_0 + C_2)\left (1+ \fr{2p_0 t}{\beta}\right)^{\fr{2C_1}{\beta}}
\ee
with $C_1$ and $C_2$ depending on $\|f\|_{L^\infty(0,T; L^\infty(\mathbb{T}^2))}$, $p_0$, $C_K$ and $\beta$.  
Thus
\[
m(t) \le m_0 \left(1+ \fr{2p_0 t}{\beta}\right)^{\fr{2C_1}{\beta}} + \fr{C_2}{C_1}\left[\left (1+ \fr{2p_0 t}{\beta}\right)^{\fr{2C_1}{\beta}} -1\right].
\]
Note that $ {p_0}/{p(t)}=  1 + 2p_0\beta^{-1} t$ is increasing on $[0,t_*]$ from $1$ to ${p_0}/{p(t_*)}=p_0$.
Consequently
\be
\|\theta(t)\|_{2p(t)} \leq C_1\|\theta_0\|_{2p_0}^{p_0}  + C_2
\ee
where the constants $C_1$ and $C_2$ have been redefined but the dependence on parameters is the same. As $p(t)\in[1,p_0]$ for all $t\in [0,t_*]$  we have that $\|\theta(t)\|_{2}\leq \|\theta(t)\|_{2p(t)}$ and we obtain
\be
\sup_{t\in [0,t_*]}   \|\theta(t)\|_{2}  \leq C_1 \|\theta_0\|_{2p_0}^{p_0} + C_2,
\ee which completes the proof. 
\end{proof}

A similar idea to our  Lemma \ref{scalalemm} was used in \cite{EJ17}, Lemma 3. We apply our two lemmas to the two dimensional linearized Euler and Navier-Stokes equations to obtain uniform boundedness of vorticity gradients for short time.
\begin{lemma}\label{lem} Fix $\ell>0$ and let  $\omega_\ell$ and $\omega_\ell^\nu$ solve \eqref{omell} and \eqref{omellnu} respectively.  Then there exists a constant $C_*$ and a constant $C_\ell<\infty$ depending only on $\ell$, the forcing norm  $\|g\|_{L^\infty(0,T;L^\infty(\mathbb{T}^2))}$, and the uniform bound on solutions given in (\ref{omegafty})  such that for $T_*\le (C_*  \Omega_{\infty})^{-1}$, we have that  
\be 
\sup_{t\in [0,T_*]} \left(\|\omega_\ell(t)\|_{H^1} + \|\omega_\ell^\nu(t)\|_{H^1}  \right)\leq C_\ell.
\ee
\end{lemma}

\begin{proof}[Proof of Lemma \ref{lem}]
We focus on proving a viscosity independent bound for $ \|\omega_\ell^\nu(t)\|_{H^1}$.  The proof for $\|\omega_\ell(t)\|_{H^1}$ is the same, setting $\nu=0$. We show that $|\nabla \omega_\ell^\nu|$ obeys \eqref{scalareqn}. Differentiating  \eqref{omellnu}, we find
\be
(\partial_t + u^\nu\cdot \nabla)\nabla \omega_\ell^\nu+\nabla u^\nu \cdot \nabla \omega_\ell^\nu =\nu \Delta (\nabla \omega_\ell^\nu)+ \nabla (\varphi_\ell *g).
\ee
A standard computation shows that $|\nabla \omega_\ell^\nu|$ satisfies
\be
(\partial_t + u^\nu\cdot \nabla-\nu \Delta) |\nabla \omega_\ell^\nu| \leq |\nabla u| |\nabla \omega_\ell^\nu|+ | \nabla (\varphi_\ell *g)|
\ee
which is a particular case of the scalar inequality \eqref{scalareqn} with $\theta=|\nabla \omega_\ell^\nu|$, initial data $\theta_0=|\nabla (\varphi_\ell *{\omega}_0^\nu)|\in L^\infty(\mathbb{T}^2)$ and forcing $f= | \nabla (\varphi_\ell *g)|\in L^\infty(0,T;L^\infty(\mathbb{T}^2))$, as claimed.  Applying Lemma \ref{scalalemm}, we find that for any $p>1$ (e.g. $p=2$) we have
\begin{align}\nonumber
\sup_{t\in [0,T_*]}   \|\omega_\ell^\nu(t)\|_{H^1} 
&=C_1 \frac{1}{\ell^p}\left( \int_{\mathbb{T}^2} | {\omega}_0^\nu*(\nabla\varphi)_\ell |^{2p} \rmd x\right)^{1/2}+C_2\\ \label{finalbnd}
& \lesssim C_\ell \|{\omega}_0^\nu\|_{L^\infty(\mathbb{T}^2)}^p\lesssim C_\ell \Omega_{\infty}^p.
\end{align}
The constant $C_\ell$ depends on $\Omega_{\infty}$.  It diverges with the mollification scale $\ell$, through the prefactor ${\ell^{-p}}$ and through the dependence on  $\|\nabla(\varphi_\ell * g)\|_{L^\infty} \lesssim  \ell^{-1} \|g\|_{L^\infty}$.  The important point however is that \eqref{finalbnd} holds uniformly in viscosity, completing the proof.
\end{proof}

We return now now to the proof of the main theorem. Using Lemma \ref{lem}, the difference energy obeys
\begin{align} \nonumber
\frac{\rmd}{\rmd t} \|\omega_\ell^\nu- \omega_\ell\|_{L^2(\mathbb{T}^2)}^2 &= -\int_{\mathbb{T}^2} (u^\nu-u) \cdot \nabla \omega^\nu_\ell (\omega_\ell^\nu- \omega_\ell) \rmd x \\ \nonumber
&\qquad - \nu  \int_{\mathbb{T}^2}|\nabla \omega_\ell^\nu|^2 \rmd x +  \nu  \int_{\mathbb{T}^2} \nabla \omega_\ell^\nu\cdot \nabla \omega_\ell \rmd x\\ \nonumber
&\leq 4 \Omega \|u^\nu-u\|_{L^{2}} \| \nabla \omega_\ell^\nu\|_{L^{2}}  + \nu \| \nabla \omega_\ell^\nu\|_{L^2} \| \nabla \omega_\ell\|_{L^2} \\
&\lesssim C_\ell   \|u^\nu-u\|_{L^\infty(0,T;L^{2}(\mathbb{T}^2))} + \nu  C_\ell ^2.
\end{align}
Integrating we find
\be\label{omdiff}
\|\omega_\ell^\nu- \omega_\ell\|_{L^2}^2\lesssim \|\varphi_\ell*( {\omega}_0^\nu- {\omega}_0)\|_{L^2}^2+  C_\ell  T   \|u^\nu-u\|_{L^\infty(0,T;L^{2}(\mathbb{T}^2))} + \nu  C_\ell ^2 T .
\ee

To conclude the proof we must show that, at fixed $\ell>0$, we have $\lim_{\nu\to 0} \|\omega_\ell^\nu- \omega_\ell\|_{L^2(\mathbb{T}^2)}=0$.  
Recall that by our assumption \eqref{limidl2} we have that $\lim_{\nu \to 0} \|\omega^\nu_0 - \omega_0\|_{L^2(\mathbb{T}^2)}\to 0$. Thus we need only establish strong convergence of the velocity in $L^2(0,T;L^2(\mathbb{T}^2))$.   If $g=0$ and $u^\nu_0=u_0$, this is a consequence of Theorem 1.4 of \cite{Chemin}.  Below is a generalization of \cite{Chemin} which applies in our setting and is proved by a different argument.

\begin{lemma}{\la{chem}} 
Let $\omega_0\in L^{\infty}(\mathbb T^2)$. There exist constants $U$, $\Omega_2$ and $K$ (see below (\ref{U}), (\ref{omegap}), (\ref{K})) depending on norms of the initial data and of the forcing such that
the difference  $v = u^{\nu}-u$ of velocities of solutions (\ref{eu}) and (\ref{nse}) obeys
\be
\|v(t)\|_{L^2}^2 \le 3U^2K^{\fr{5(t-t_0)\Omega_{\infty}}{\gamma}}
\left (\fr{\|v(t_0)\|_{L^2(\mathbb T^2)}^2}{U^2} + \gamma \fr{\Omega_2^2}{U^2\Omega_{\infty}}\nu \right)^{1-\fr{5(t-t_0)\Omega_{\infty}}{\gamma}}
\la{vineq}
\ee
for all $0\le t_0\le  t $.
By iterating the above, we obtain
\be\label{vineqLT}
\|v(t)\|_{L^2}^2 \leq 20U^2K^{1-e^{-{10 t \Omega_{\infty}}/{\gamma}}}
\left (\fr{\|v(0)\|_{L^2(\mathbb T^2)}^2}{U^2} +  \gamma \fr{ \Omega_2^2}{U^2\Omega_{\infty}}\nu \right)^{e^{-\fr{10 t \Omega_{\infty}}{\gamma}}}
\ee 
provided that $\|v(0)\|_{L^2(\mathbb T^2)}^2+ {\gamma}\nu \Omega_2^2/{\Omega_{\infty}}\leq 9K U^2$.
\end{lemma}

\begin{rem}[Continuity of Solution Map]
At zero viscosity, Lemma \ref{chem} establishes H\"{o}lder continuity of the Yudovich (velocity) solution map.  Specifically,  denoting $u_t:= S_t^v(u_0)$ and setting $\nu=0$, a consequence of  Lemma \ref{chem} is that $\| S_t^v(u_0)- S_t^v(u_0')\|_{L^2(\mathbb{T}^2)} \leq C \|u_0-u_0'\|_{L^2(\mathbb{T}^2)}^{\alpha(t)}$ where $\alpha(t):= e^{-ct} $  and $c, C>0$ are appropriate constants.  This fact is used to prove Corollary \ref{contcor}. {It is worth further remarking that the condition on the data  $\|v(0)\|_{L^2(\mathbb T^2)}^2 \leq 9K U^2$  required for the above estimate to hold is $O(1)$ (data need not be taken very close).}
\end{rem}

\begin{proof}[Proof of Lemma \ref{chem}]
The proof proceeds in two steps.
\\

\vspace{-3mm}
\noindent \textbf{Step 1: Short time bound.} 
The proof of the lemma starts from the equation obeyed by the difference
$v$,
\[
\pa_t v + u^{\nu}\cdot\na v + v\cdot\na u + \na p = \nu\Delta v + \nu \Delta u
\]
leading to the inequality
\be
\fr{d}{dt}\|v\|_{L^2}^2  + \nu \|\na v\|^2_{L^2} \le \nu \|\na u\|^2_{L^2} +
2 \int |\na u| |v|^2\rmd x
\la{env}
\ee
which is a straightforward consequence of the equation, using just integration by parts. 
We use the bound $\Omega_{\infty}$ (\ref{omegafty}) for the vorticity of the Euler solution. We also use a bound for the $L^2$ norms
\be
\sup_{0\le t\le T}\left(\|u^{\nu}(t)\|_{L^2(\mathbb T^2)} + \|u(t)\|_{L^2(\mathbb T^2)}\right)\le U,
\la{U}
\ee
which is easily obtained from energy balance. We use also bounds for $L^p$ norms of vorticity, 
\be
\Omega_p  = \sup_{0\le t\le T}\|\omega(t)\|_{L^p(\mathbb T^2)} \le \Omega_\infty.
\la{omegap}
\ee
We split the integral
\[
 \int |\na u| |v|^2  \rmd x = \int_B |\na u| |v|^2\rmd x + \int_{\mathbb T^2\setminus B} |\na u| |v|^2\rmd x
\]
where 
\[
B = \{ x\left |\right.\; |v(x,t)| \ge MU\}
\]
with $M$ to be determined below. Although $B$ depends in general on time, it has small measure if $M$ is large, 
\[
\left | B\right | \le M^{-2}.
\]
The constant $M$ has dimensions of inverse length.  We bound
\be\label{termbd}
2\int_B|\na u||v|^2\rmd x \leq  2 \| \na u\|_{L^2} \| v\|_{L^4}^2\leq 2 |B|^{\fr{1}{4}}\| \na u\|_{L^4} \|v(t)\|_{L^4}^2
\ee
where we used $\int_B|\na u|^2dx \le |B|^{\fr{1}{2}}\| \na u\|_{L^4}^2$.
We now use the fact that we are in Yudovich class and Ladyzhenskaya inequality to deduce
\[
\|v(t)\|_{L^4}^2 \le C\|v(t)\|_{L^2}[\|\omega_0\|_{L^2} + \|g\|_{L^1(0,T; L^2)}] \le CU\Omega_2
\]
and we use also
\[
\|\na u\|_{L^4}\le [C\|\omega_0\|_{L^4} + \|g\|_{L^1(0,T; L^4)}] = \Omega_4
\]
to bound \eqref{termbd} by 
\be
2\int_B|\na u||v|^2\rmd x \le CU\Omega_2\Omega_4M^{-\fr{1}{2}},
\la{smallb}
\ee
We nondimensionalize by dividing by $U^2$ and we multiply by $\beta = {\gamma}/{\Omega_{\infty}}$.
The quantity 
\be
y(t) = \fr{\|v(t)\|_{L^2{(\mathbb T^2)}}^2}{U^2}
\la{y}
\ee
obeys the inequality
\be
\beta\fr{\rmd y}{\rmd t} \le \beta\nu\fr{\Omega_2^2}{U^2} + C\beta\Omega_4\fr{\Omega_2}{U}M^{-\fr{1}{2}} + 2\int_{\mathbb T^2\setminus B} \beta|\na u|\frac{|v|^2}{U^2}\rmd x.
\la{intery}
\ee
We write the term
\be
2\int_{\mathbb T^2\setminus B} \beta |\na u| |v|^2 U^{-2} \rmd x = 2\int_{\mathbb T^2\setminus B}(\beta|\na u| + \log\epsilon + \log\fr{1}{\epsilon})|v|^2 U^{-2} \rmd x\ee
with $\epsilon$ (with units of inverse area) to be determined below.  We use the inequality (\ref{elemIq}) and Lemma \ref{explemm} with
\[
a = \beta |\na u| + \log\epsilon, \qquad b= \fr{|v|^2}{U^2}
\]
 to deduce
\be
2\int_{\mathbb T^2\setminus B} \beta |\na u| |v|^2 U^{-2}\rmd x\le  2\epsilon C_K + 2\log\fr{M^2}{\epsilon} y(t).
\la{largeb}
\ee
Inserting (\ref{largeb}) in (\ref{intery}) we obtain
\be
\beta\fr{\rmd y}{\rmd t} \le F + \log\left(\fr{M^2}{\epsilon}\right )y(t)
\la{yneqdt}
\ee
with
\be
F =  \beta\nu\fr{\Omega_2^2}{U^2} + C\beta\Omega_4\fr{\Omega_2}{U}M^{-\fr{1}{2}} +  2\epsilon C_K.
\la{f}
\ee
Note that $F$ and $\fr{M^2}{\epsilon}$ are nondimensional.
From (\ref{yneqdt}) we obtain immediately
\be
y(t) \le \left(\fr{M^2}{\epsilon}\right)^{\fr{t-t_0}{\beta}} y(t_0) + \fr{F}{\log\left(\fr{M^2}{\epsilon}\right )}\left(\left(\fr{M^2}{\epsilon}\right)^{\fr{t-t_0}{\beta}} -1\right).
\la{yneq}
\ee
We choose $M$ such that
\be
 C\beta\Omega_4\fr{\Omega_2}{U}M^{-\fr{1}{2}} = \beta\nu\fr{\Omega_2^2}{U^2} + y(t_0)
\la{Mchoice}
\ee
and  we choose $\epsilon$ such that
\be
 2\epsilon C_K = \beta\nu\fr{\Omega_2^2}{U^2} + y(t_0).
\la{epsilonchoice}
\ee
These choices imply
\be
F = 3\beta\nu\fr{\Omega_2^2}{U^2} + 2y(t_0).
\la{fexplicit}
\ee
Then we see that
\be
\Gamma = \fr{M^2}{\epsilon} = 2C_K\left(C\beta\Omega_4\fr{\Omega_2}{U}\right)^4\times \left (\beta\nu\fr{\Omega_2^2}{U^2} + y(t_0)\right)^{-5}.
\la{Gamma}
\ee
Taking without loss of generality $\log \Gamma \ge 1$, we have  from (\ref{yneq}) 
\begin{align}\nonumber
y(t) &\le 3\left (y(t_0) + \beta \nu \fr{\Omega_2^2}{U^2}\right)\Gamma^{\fr{t-t_0}{\beta}}\\
& \le 3\left (y(t_0) + \beta \nu \fr{\Omega_2^2}{U^2}\right)^{1-\fr{5(t-t_0)}{\beta}}\times \left(2C_K\left(C\beta\Omega_4\fr{\Omega_2}{U}\right)^4\right)^{\fr{5(t-t_0)}{\beta}}.
\la{ygamma}
\end{align}
Recalling that $\beta = {\gamma}/{\Omega_{\infty}}$ and
denoting the nondimensional constant
\be
K = 2C_K\left(C\beta\Omega_4\fr{\Omega_2}{U}\right)^4
\la{K}
\ee
we established 
\be
\fr{\|v(t)\|^2}{U^2} \le 3 K^{\fr{5(t-t_0)\Omega_{\infty}}{\gamma}}\left (\fr{\|v(t_0)\|_{L^2(\mathbb T^2)}^2}{U^2} + \beta \nu \fr{\Omega_2^2}{U^2}\right)^{1-\fr{5(t-t_0)\Omega_{\infty}}{\gamma}}.
\la{vbound}
\ee
Thus, we established \eqref{vineq}.
\\

\noindent \textbf{Step 2: Long time bound.}
With  \eqref{vineq} established, we now prove \eqref{vineqLT}.  Let  $c=  {5\Omega_{\infty}}/{\gamma}$, $\Delta t= 1/2c$ and  $t_i = t_{i-1}+\Delta t$ and $a_i = \|v(t_i)\|_{L^2}^2/U^2$  for $i\in \mathbb{N}$.
Then \eqref{vineq} states
\be
a_i \leq C_1 \left(a_{i-1} + C_2\nu \right)^{1/2}, \qquad i=1, 2, \dots
\la{aineq}
\ee
with $C_1= 3K^{\fr{5\Omega_{\infty}}{2c\gamma}}=3K^{\fr{1}{2}}$ and $C_2=  \beta \fr{\Omega_2^2}{U^2}$. 
We set
\be
\delta_n = \fr{a_i+ C_2\nu}{C_1^2}
\la{deltan}
\ee
and observe that (\ref{aineq}) is
\be
\delta_n \le \sqrt{\delta_{n-1}} + \widetilde{\nu}
\la{deltanineq}
\ee
where 
\be
\widetilde{\nu} = \fr{C_2\nu}{C_1^2}
\la{widetlidenu}
\ee
is a nondimensional inverse Reynolds number. It follows then by induction that
\be
\delta_n \le (\delta_0)^{2^{-n}} + \sum_{i=0}^{n-1}(\widetilde{\nu})^{2^{-i}}.
\la{deltanind}
\ee
Indeed, the induction step follows from 
\be
\delta_{n+1} \le \sqrt{\delta_n} + \widetilde{\nu}
\la{indstep}
\ee
and the subadditivity of $\lambda\mapsto \sqrt{\lambda}$. If 
\be
\widetilde{\nu} \le \fr{1}{\sqrt{5}-1}
\la{widetildecond}
\ee
then the iteration (\ref{deltanineq}) starting from $0<\delta_0< r$ where $r$ is the positive root of the equation $x^2-x-\widetilde{\nu} = 0$, remains in the interval $(0,r)$, and for any  $n$, $\delta_n$ obeys (\ref{deltanind}). We observe that 
\be
 \sum_{i=0}^{n-1}(\widetilde{\nu})^{2^{-i}}
= (\widetilde{\nu})^{2^{-n+1}}\left (1 + \cdots + (\widetilde{\nu})^{2^{n-1}}\right)\le \fr{1}{1-\widetilde{\nu}}(\widetilde\nu)^{2^{-n+1}}
\la{smallnutilde}
\ee
and therefore (\ref{vineqLT}) follows from (\ref{deltanind}). 
We note that the iteration defined with equality in (\ref{deltanineq}) converges as $n\to\infty$ to $r$. {Fixing any $t>0$ and letting $n= \lceil t/\Delta t\rceil = \lceil 2ct\rceil = \lceil{10 t \Omega_{\infty}}/{\gamma}\rceil$ establishes the bound.}
\end{proof}

Due to assumption \eqref{limidl2} we have that $\lim_{\nu \to 0} \|u^\nu_0 - u_0\|_{L^2(\mathbb{T}^2)}\to 0$.  Lemma \ref{chem} then allows us to conclude from \eqref{omdiff} that
$\lim_{\nu \to 0}\sup_{t\in [0,T_*]} \|\omega_\ell^\nu- \omega_\ell\|_{L^2(\mathbb{T}^2)}\to0$ at fixed $\ell>0$ and the proof of proposition \ref{prop} is complete.
\end{proof}

With the Proposition proved, the proof of the strong convergence of the vorticity in $L^p$ statement in the theorem is established. To obtain convergence of the distribution functions, see Thm 3.6 in \cite{CW2}.
\end{proof}

\begin{proof}[Proof of Proposition \ref{propReg}]
This proof makes use of the the stochastic Lagrangian representation for Navier--Stokes solutions \cite{CI08}, together with the uniform--in--$\nu$ boundedness of vorticity.  
In light of the Lagrangian representation \eqref{traj}, \eqref{stochRep}, the key ingredient of propagating some degree of fractional regularity on the vorticity is the (uniform) H\"{o}lder regularity of the  inverse flow $A_t$.
Since the diffusion coefficients on the additive noise on \eqref{traj} are spatially constant, it follows that the results of Chapter 3 of \cite{BCD11} hold realization-by-realization for the stochastic flow $X_t$ and its inverse $A_t$, uniformly in viscosity. 
 This gives uniform bounds on the separation of two trajectories driven by the same realization of Brownian noise, independent of viscosity, thereby establishing spatial H\"{o}lder regularity of the flow. Although straightforward, we include a proof of this statement for completeness.
 \begin{prop}\label{flowreg}
 There exists a unique measure-preserving stochastic flow of homeomorphisms solving \eqref{traj}. This flow and the back-to-labels map are continuous flows $X, A$ which for all $t\in [0,T]$ are uniformly--in--$\nu$ of the class $C^{\alpha(t)}(\mathbb{T}^2)$ with $\alpha(t)= \exp(-Ct/\beta)$ with constants defined in \eqref{loglip}.
 \end{prop}
 \begin{proof}[Proof of Proposition \ref{flowreg}]
 We employ the log-Lipshitz property of $u^\nu$, i.e. there exists an absolute constant $C>2$ such that one has the following uniform-in-viscosity estimate
 \be\label{loglip}
 |u^\nu(x,t)-u^\nu(y,t)|\leq   \frac{C}{\beta} d(x,y) \ln\left(\frac{CC_K}{ d(x,y)^2}\right), \qquad \forall x,y\in \mathbb{T}^2,
 \ee
where $\beta$ and $C_K$ are the constants in Lemma \ref{explemm} which depend only on $\|\omega_0\|_{L^\infty}$.   See Lemma A.1 of \cite{BN19}. Here  $d(x,y) := \min \{ |x-y-k| \ : \ k\in \mathbb{Z}^d, \ |k|\leq 2\}$ is the geodesic distance on the torus upon the identification  $\mathbb{T}^d=[0,1)^d$.   
Now, due to the spatial uniformity of the noise on the trajectories 
\be
\rmd X_t(x) = u(X_t(x),t)\rmd t+\sqrt{2\nu}\ \rmd W_t , \qquad X_0(x)=x,
\ee
 we find that the difference  has no martingale part and satisfies
\be
 \rmd\left(X_t(x)- X_t(y)\right) =\left( u^\nu(X_t(x),t)- u^\nu(X_t(y),t)\right)\rmd t.
\ee
Upon integration, we obtain the inequality
\begin{align}\nonumber
d\left(X_t(x),X_t(y)\right) &\leq d(x,y) +\frac{C}{\beta} \int_0^t d\left(X_s(x),X_s(y)\right) \ln\left(\frac{CC_K}{ d\left(X_s(x),X_s(y)\right)^2}\right)\rmd s,  \label{holdertraj}
\end{align}
for all $x,y\in \mathbb{T}^2$.
The solution of this integro-inequality  (with a possibly larger constant $C$) is
\be\label{ineqX}
 d\left(X_t(x),X_t(y)\right)  \leq (C C_K)^{1+ e^{-Ct/\beta}} d(x,y)^{ e^{-Ct/\beta}} \quad \text{a.s.}.
\ee
  Since the bound holds almost surely, this says that the map $X_t(\cdot)$ is  H\"{o}lder continuous $C^{\alpha(t)}(\mathbb{T}^2)$ with  $\alpha(t)= e^{-Ct/\beta}$ as claimed.  We remark that deterministic trajectories in a log-Lipshitz field satisfying \eqref{loglip} satisfy precisely the same upper bound \eqref{ineqX}.
 
 To obtain H\"{o}lder regularity of the back-to-labels map, it suffices to note that $A_t$ can be identified with the backwards flow $X_{t,0}$ which solves the following backward stochastic differential equation 
 \be
 \hat{\rmd} X_{t,s}(x) = u(X_{t,s}(x) ,s) \rmd s +\sqrt{2\nu} \ \hat{\rmd}  \widehat{W}_s, \qquad X_{t,t}(x)=x,
 \ee
 where the $ \hat{\rmd} $ indicates the the backward differential and $\widehat{W}_s=W_{t-s}- W_t$ is a Brownian motion adapted to the backward filtration $\hat{\mathcal{F}}_s^t:= \sigma\{ \widehat{W}_u, u\in [0,s]\}$. For a discussion of backward It$\bar{{\rm o}}$ equations, see e.g. \cite{Kunita}. With this identification, one finds as above that for any $t>0$ and all $s\in [0,t]$ one has
\be\label{ineqA}
 d\left(X_{t,s}(x),X_{t,s}(y)\right)  \leq (C C_K)^{1+ e^{-C(t-s)/\beta}} d(x,y)^{ e^{-C(t-s)/\beta}} \quad \text{a.s.}.
\ee
By setting $s=0$ we find that $A_t=X_{t,0}$ satisfies the same estimate  \eqref{ineqX} as $X_t$ and therefore is H\"{o}lder continuous with the same exponentially decaying exponent.
 \end{proof}

Proceeding forward to obtain uniform bounds we wish to make use of the representation formula  \eqref{traj}, \eqref{stochRep}.  This requires some regularity on the initial condition, so we replace $\omega_0\in L^\infty(\mathbb{T}^2)$ with a mollification of it $\omega_0*\varphi_\ell \in C^\infty (\mathbb{T}^2)$ for $\ell>0$.  All the bounds will be manifestly independent of $\ell$ which can be taken to zero at the end, so we  simplify the notation by writing ``$\omega_0$".

 We continue by following closely the proof of Theorem 3.32 of \cite{BCD11}.  In particular, we introduce the space $F_p^s(\mathbb{T}^d)$ (which belongs to the family of Triebel--Lizorkin spaces $F_p^s= F_{p,\infty}^{s}$ provided $p>1$) that is comprised of measurable functions $f\in L^p(\mathbb{T}^d)$ which are finite in the seminorm
  \begin{align}\nonumber
 [f]_{F_p^s} := \inf_{g\in L^p(\mathbb{T}^d) } \Big\{ \|g\|_{L^p(\mathbb{T}^d)} \ : &\ \ |f(x)-f(y)|\leq d(x,y)^\alpha(g(x)+g(y)),\\
 &\qquad\qquad\qquad\qquad \qquad  \ \forall\  x,y\in \mathbb{T}^d\Big\}<\infty\label{Fseminorm}
 \end{align}
 where $d(x,y)$ is the distance function on the torus defined above. See Definition 3.30 of \cite{BCD11}.
 The key of the argument is to understand how composition with a (uniformly) H\"older continuous stochastic diffeomorphism provided by Prop. \ref{flowreg} operate on $F_p^s$.  Using the stochastic representation \eqref{stochRep},
 \be 
\omega^\nu(t) = \mathbb{E}\left[\omega_0\circ A_t \right],
\ee
 Jensen's inequality, H\"older continuity of the back-to-labels map and the fact that $\omega_0\in F_p^s$ we have
 \begin{align}\nonumber
 \frac{|\omega^\nu(x,t)-\omega^\nu(y,t)|}{d(x,y)^{s\alpha}} &=  \frac{|\mathbb{E}[\omega_0(A_t(x))-\omega_0(A_t(y)) ]|}{d(x,y)^{s\alpha}}\\ \nonumber
 &\leq  \mathbb{E}\left[ \frac{|\omega_0(A_t(x))-\omega_0(A_t(y))|}{d(x,y)^{s\alpha}} \right] \\ \nonumber
 &=  \mathbb{E}\left[ \frac{|\omega_0(A_t(x))-\omega_0(A_t(y))|}{d\left(A_t(x),A_t(y)\right)^s} \frac{d\left(A_t(x),A_t(y)\right)^s}{d(x,y)^{s\alpha}} \right]\\  \label{Finewt}
 &\leq \|A_t\|_{C^{\alpha}}^s \mathbb{E} [g(A_t(x))+g(A_t(y))]
 \end{align}
 for any $g\in L^p(\mathbb{T}^2)$, where we used that  $\omega_0\in F_p^s$ together with the definition \eqref{Fseminorm}.  Letting $ \tilde{g}(x):= \mathbb{E} [g(A_t(x))]$.  Note that, since $A_t$ is measure preserving and Jensen's inequality, we have $\| \tilde{g}\|_{L^p(\mathbb{T}^2)}\leq \| {g}\|_{L^p(\mathbb{T}^2)}<\infty$.  Thus $\tilde{g}\in L^p(\mathbb{T}^2)$ and it follows  by linearity of the expectation that the the right-hand-side of \eqref{Finewt} is a $L^p$ function. This shows
 \be\label{Fbnd}
[\omega^\nu(t)]_{F_p^{s(t)}} \leq (C K)^{s(1+ e^{-Ct/\beta})} [\omega_0]_{F_p^{s}}, \qquad s(t)= s \exp(-Ct/\beta)
 \ee
where we used the explicit bound on H\"{o}lder norm computed in \eqref{ineqX}.  The bound \eqref{Fbnd} holds uniformly in viscosity.  In order to connect to some $B^{s}_{p,\infty}$ (which is a
larger space) we need to use an embedding for the initial data
\be
B^{s_3}_{p,\infty} \subset B^{s_2}_{p,1}\subset W^{s_2,p}\subset F^{s_1}_p
\ee
with $s_3>s_2>s_1$. 
 The proposition follows from Lemma 3.31 of \cite{BCD11}, which shows that the Triebel--Lizorkin spaces are continuously embedded in the Besov spaces, i.e. $F_p^s(\mathbb{T}^d) \hookrightarrow B_{p,\infty}^{s}(\mathbb{T}^d)$.

 \end{proof}

\begin{proof}[Proof of Corollary \ref{cor}]
We need the following elementary Lemma
\begin{lemma}\label{lemma}
For any $s>0$ and $f\in B_{2,\infty}^{s}(\mathbb{T}^d)$, the following inequality holds for all $0<s'<s$
\be\label{besovineq}
\|f\|_{L^2(\mathbb{T}^d)}  \leq \|f\|_{H^{-1}(\mathbb{T}^d)}^{s'/(1+s')}\|f\|_{B^{s}_{2,\infty} (\mathbb{T}^d)}^{1/(1+s')}.
\ee
\end{lemma}
\begin{proof}
First note that the interpolation inequality 
\be\nonumber
\|f\|_{L^2(\mathbb{T}^d)}  \leq \|f\|_{H^{-1}(\mathbb{T}^d)}^{s'/(1+s')}\|f\|_{H^{s'} (\mathbb{T}^d)}^{1/(1+s')}
\ee
 which follows from Holder inequality and the Fourier definition of the Sobolev norm. The claim follows from the embedding $B_{p, q}^{s}(\mathbb{T}^d) \subset B_{p,q'}^{s'}(\mathbb{T}^d)$ for $s'<s$ and any $q',q$ (see \S 2.3.2 of \cite{Treibel}) and the identification  $H^{s}:= B^{s}_{2,2}$.
\end{proof}
Proceeding with the proof, applying Lemma \ref{lemma} for all $t\in [0,T]$ we have
\begin{align*}
\| \omega^\nu(t) -\omega(t)\|_{L^2(\mathbb{T}^2)} &\leq 
\|\omega^\nu(t) -\omega(t)\|_{H^{-1}(\mathbb{T}^2)}^{\frac{s'}{1+s'} } \| \omega^\nu(t) -\omega(t)\|_{B_{2,\infty}^{s(t)}(\mathbb{T}^2)}^{\frac{1}{1+s'} } \\
&\lesssim \sup_{t\in [0,T]} \|u^\nu(t)-u(t)\|_{L^2(\mathbb{T}^2)}^{\frac{s(t)}{1+s(t)}- } 
\end{align*}
 for any $s'< s(t):= {s} \exp(-CT\|\omega_0\|_{\infty})$. 
In the above, we appealed to Proposition \ref{propReg} to establish uniform--in--$\nu$ boundedness of the solution $\omega^\nu$ in the space  $L^\infty(0,t; B_{2,\infty}^{s(t)}(\mathbb{T}^2))$.  We now use Lemma \ref{chem} to conclude
\begin{align}\nonumber
\|\omega^\nu(t) -\omega(t)\|_{L^p(\mathbb{T}^2)} &\leq \|\omega^\nu -\omega\|_{L^\infty(0,T;L^\infty(\mathbb{T}^2))}^{\frac{p-2}{p}}\|\omega^\nu(t) -\omega(t)\|_{L^2(\mathbb{T}^2)}^{\frac{2}{p}}\\
&  \lesssim \sup_{t\in [0,T]} \|u^\nu(t)-u(t)\|_{L^2(\mathbb{T}^2)}^{\frac{2s(t)}{p(1+s(t))}- }\lesssim  (\nu T)^{\frac{ s\exp(-2CT\|\omega_0\|_{\infty})}{p(1+ s \exp(-CT\|\omega_0\|_{\infty}))} -}.
\end{align}
This completes our proof.
\end{proof}

\begin{rem}
The stochastic Lagrangian representation of the vorticity offers also an expression for the enstrophy dissipation as the variance of the (randomly sampled) initial data
\be
\nu \int_0^t \int_{\mathbb{T}^2} |\nabla \omega^\nu (t',x)|^2 \rmd x \rmd t' = \frac{1}{2} \int_{\mathbb{T}^2}{\rm Var} \left[\omega_0^\nu(A_t(x)) \right] \rmd x. 
\ee
The above is a special case of the Lagrangian fluctuation dissipation relation for active scalars derived in \cite{DE17}.
This relation is easily generalized to incorporate the effect of body forces.  A consequence of our Theorem \ref{Thm} is that the enstrophy dissipation vanishes in the high Reynolds number limit, forcing also the variance to become zero.  Thus, there is no ``spontaneous stochasticity" of Lagrangian trajectories in the vanishing viscosity limit for 2d Navier-Stokes with initial data in the Yudovich class.
\end{rem}

\section{Discussion}

Predicting the long-time vortex structures in two-dimensional turbulence is of long standing interest, starting with the work on dynamics of point vortices by Onsager \cite{Onsager}.  There have been
a number of theories developed to this effect.   We briefly review the celebrated mean-field theory of  Miller \cite{Miller} and Robert
\cite{Robert} to give context to our result. The idea is to describe
an equilibrium configuration $\omega_{eq}$ satisfying
\be
\omega(t) \ \ \xrightarrow[]{t \to \infty} \ \  \omega_{eq} 
\ee
in some sense. If $\omega(t)$ is an Euler path with bounded initial vorticity, then one has the information
\begin{enumerate}
\item conservation of energy: 
\be\label{enercons}
\|u(t)\|_{L^2(\mathbb{T}^2)} = \|u_0\|_{L^2(\mathbb{T}^2)},
\ee
\item conservation of  vorticity ``casmirs": for any continuous $f$,
\be\label{ifinv}
I_f:= \int_{\mathbb{T}^2} f(\omega(x,t)) \rmd x = \int_{\mathbb{T}^2} f(\omega_0(x)) \rmd x.
\ee
\end{enumerate}
For long-time limits of Euler flows, there is a natural candidate object to describe $ \omega_{eq}$.  In particular, provided only $\omega_0\in L^\infty(\mathbb{T}^2)$, then $\omega(t)\in L^\infty(\mathbb{T}^2)$ is the unique solution of Euler \cite{Yudovich} and  in the weak--$*$ sense
\be\label{Onsagertheory}
\lim_{n\to\infty} \int_{\mathbb{T}^2} \varphi(x) \omega(x,t_n) \rmd x = \int_{\mathbb{T}^2} \varphi(x) \bar{\omega}(x)\rmd x,  \qquad \forall \varphi\in L^1(\mathbb{T}^2)
\ee
for some $\bar{\omega}\in L^\infty(\mathbb{T}^2)$ and some subsequence $t_n\to \infty$ as $n\to \infty$.
However, large oscillations can remain in this limit.  In particular, the above convergence  does not imply for all continuous functions $f$ that $f(\omega(x,t_n))$ converges to $f(\bar{\omega}(x))$ 
in the same sense, so it is not clear how the information \eqref{enercons} and \eqref{ifinv} can be retained and in what sense.  On the other hand, the fundamental theorem of Young measures guarantees 
\be\label{youngmeasurelim}
\lim_{n\to\infty} \int_{\mathbb{T}^2} \varphi(x) f(\omega(x,t_n))  \rmd x= \int_{\mathbb{T}^2} \varphi(x) \int_{-M}^M f(y) \nu_x(\rmd y) \rmd x 
\ee
with $M=\|\omega_0\|_{L^\infty(\mathbb{T}^2)}$.
Note that, having introduced the Young measure $\nu_x(\rmd y)$, the convergence \eqref{Onsagertheory} holds with 
\be \label{barom}
 \bar{\omega}(x)=  \int_{-M}^M y \nu_x(\rmd y), \qquad \forall f\in C([-M,M]).
\ee
Kraichnan developed a theory for the equilibrium distribution $ \bar{\omega}$ discarding most of the information on the casmirs, keeping only conservation of energy and enstrophy \cite{Kraichnan}.  However, it was since recognized that invariants involving higher powers of vorticity should not be neglected on compact domains such as $\mathbb{T}^2$. 
In order to retain as much information about the Euler solution as possible, Miller \cite{Miller} and Robert
\cite{Robert} independently suggested that the long-time vorticity distribution resulting from freely decaying two-dimensional turbulence is a Young measure of the form
\be \label{MRmeasure}
\nu_x(\rmd y) = \rho(x,y) \rmd y.
\ee
These Young measures have the property that their marginal distribution is the (initial) vorticity distribution function \eqref{distconv}, which is left invariant under the Euler flow.  Thus, if a measure \eqref{MRmeasure} with 
the above property can be constructed such that also the energy associated to $\bar{\omega}$ equals that of $\omega_0$, then the information on all ideal invariants is retained at the level of the predicted equilibrium distribution. Miller and Robert provide such a construction.\footnote{We remark that the Miller--Robert theory applies for any compact domain $\Omega\subset\mathbb{R}^2$ with smooth boundary, with the  torus $\Omega=\mathbb{T}^2$ as a special case. It is worth noting that convergence of higher-order vorticity moments in the zero viscosity limit on domains with boundaries is  -- in general -- false. In fact, if the Euler velocity $u$ is not identically zero along the boundary and no-slip Navier-Stokes solutions converge to these e.g. $u^\nu\wc u$ weakly in $L^\infty(0,T;L^2(\Omega))$, then $\limsup_{\nu\to 0} \|\omega^\nu\|_{L^\infty(0,T;L^p(\Omega))}= \infty$ for all $p\in (1,\infty]$ (see Theorem 3.1 of \cite{Kelliher}).  If weak convergence fails to hold, then by Kato's energy dissipation condition we know  $\limsup_{\nu\to 0} \|\omega^\nu\|_{L^2(0,T;L^2(\Omega))}= \infty$. Thus, unless the Euler solution is identically zero on the boundary, higher moments of vorticity must diverge in the inviscid limit, presenting a great difficulty for the Miller--Robert theory as it applies to inviscid limits.}  Specifically, by a Boltzmann counting argument, they showed that the entropy associated with a given density $\rho(x,y)$ of the Young measure has a specific form.  Assuming ergodicity at long times, i.e. that the 2D Euler flow is sufficiently chaotic in phase space, they
suggested to maximize this entropy subject to the above constraints.
The prediction of the theory is the long-time distribution is
\be
\rho(x,y) = \frac{\exp\left(\beta\left[y\bar{\psi}(x)+ \mu(y)\right]\right)}{\int_{-M}^M \exp\left(\beta\left[y\bar{\psi}(x)+ \mu(y)\right]\right)\rmd y},
\ee
where the ``inverse temperature" $\beta$ and ``chemical potential" $\mu(y)$ are Lagrange multipliers to enforce energy conservation and the marginal density $\pi_{\omega_0}[\rmd y]$ respectively, and where 
the stream function $\bar{\psi}$ solves 
\be
\Delta \bar{\psi}(x) = \bar{\omega}= \frac{\int_{-M}^M y\exp\left(\beta\left[y\bar{\psi}(x)+ \mu(y)\right]\right)\rmd y}{\int_{-M}^M \exp\left(\beta\left[y\bar{\psi}(x)+ \mu(y)\right]\right)\rmd y}.
\ee
Thus, the prediction is that the expected
(average or coarsened) vorticity solves a very particular steady Euler equation
$\omega=F(\psi)$ where $\psi$ is the stream function. The function $F$ depends on the distribution $\pi_{\omega_0}[\rmd y]$ and the energy $E_0$.  
It is important to remark that conservation individual casmirs may not survive as $t\to\infty$, but that according to this theory, at a given energy $E_0$, they are forever remembered at the level of the equilibrium distribution.
Some numerical simulations have provided corroboratory evidence supporting this theory over competitive ones such as the Onsager-Joyce-Montgomery theory, at least in situations where $\omega_0$ is supported on a finite area \cite{Sommeria}. Whether or not the theory rigorously applies is open.

There are two major questions remaining about the domain of applicability of the Miller--Robert theory.  The first being whether or not 2D Euler possesses the requisite ergodicity properties to justify entropy maximization.   The second, and the one that motivates the present study, is whether the theory should apply to 2D Navier-Stokes solutions at small viscosity. This is related to the issue of anomalies in ideally conserved quantities.  For energy, there is no question since $E^\nu(t):=\frac{1}{2} \int_{\mathbb{T}^2} |u^\nu(t)|^2 \rmd x \xrightarrow[]{\nu \to 0} E_0$ for any finite time under the assumption that $\omega_0\in L^\infty(\mathbb{T}^2)$. On the other hand, it has not been clear that high-order ideal
moments such as $I_n^\nu= \int_{\mathbb{T}^2} |\omega^\nu(t)|^n \rmd x$ for $n>2$ will be conserved in the limit of zero viscosity or if there will be an associated anomaly do to fine-scale mixing of the vorticity field.  If they are not, it seems unlikely that these casmirs should be remembered at the level of the equilibrium distribution of vorticity.
Our Theorem establishes that there can be no such anomalies of higher-order invariants on any finite time interval $[0,T]$ with $T$ arbitrarily large. Thus, it shows that the dependence of $F$ on viscosity is slow which provides a partial foundation for the Miller--Robert theory as it applies to slightly viscous fluids.

 \subsection*{Acknowledgments} We would like to thank 
Helena J. Nussenzveig Lopes for insightful comments.  The research of PC was partially supported by NSF grant DMS-1713985.
Research of TD was partially supported by
NSF grant DMS-1703997.  Research of TE was partially supported by
NSF grant DMS-1817134.

\bibliographystyle{cpam}

\end{document}